\documentclass[10pt,leqno,a4paper]{article} 

\usepackage{amsmath}
\usepackage{amsthm}
\usepackage{amssymb}

\usepackage{hyperref}

\allowdisplaybreaks[3]

\theoremstyle{plain}
\newtheorem{theorem}{Theorem}

\newcommand{\Ss}{\mathcal{S}}

\newcommand{\RR}{\mathbb{R}}
\newcommand{\ZZ}{\mathbb{Z}}
\newcommand{\NN}{\mathbb{N}}
\newcommand{\CC}{\mathbb{C}}

\renewcommand{\Im}{\operatorname{Im}}
\renewcommand{\Re}{\operatorname{Re}}

\numberwithin{equation}{section}

\newcommand{\FF}[3]{{}_{2}F_{1}%
  \mathopen{}\left(\genfrac{}{}{0pt}{}{#1}{#2};#3\right)\mathclose{}}
\newcommand{\unoFFdos}[3]{{}_{1}F_{2}%
  \mathopen{}\left(\genfrac{}{}{0pt}{}{#1}{#2};#3\right)\mathclose{}}
\newcommand{\tresFFdos}[3]{{}_{3}F_{2}%
  \mathopen{}\left(\genfrac{}{}{0pt}{}{#1}{#2};#3\right)\mathclose{}}

\title{Summing Sneddon-Bessel series explicitly\thanks{This research was 
partially supported by grants PID2021-124332NB-C21 and PID2021-124332NB-C22
(FEDER(EU)/Mi\-nis\-te\-rio de Cien\-cia e Inno\-va\-ci\'on-Agen\-cia Esta\-tal de Inves\-ti\-ga\-ci\'on),
FQM-262 (Jun\-ta de Anda\-lu\-c\'ia), and E48\_20R (Go\-bier\-no de Ara\-g\'on).}}

\author{Antonio J. Dur\'an\textsuperscript{1}, Mario P\'erez\textsuperscript{2} 
 and Juan L. Varona\textsuperscript{3}
\\[6pt]
\small\textsuperscript{1}Departamento de An\'alisis Matem\'atico and IMUS,
Universidad de Sevilla,\\[-2pt]
\small 41080 Sevilla, Spain. Email: {duran@us.es}\\[6pt]
\small\textsuperscript{2}Departamento de Matem\'aticas and IUMA,
Universidad de Zaragoza,\\[-2pt]
\small 50009 Zaragoza, Spain. Email: {mperez@unizar.es}\\[6pt]
\small\textsuperscript{3}Departamento de Matem\'aticas y Computaci\'on, 
Universidad de La Rioja,\\[-2pt]
\small 26006 Logro\~no, Spain. Email: {jvarona@unirioja.es}}

\date{March 23, 2022}

\begin{document}

\maketitle

\begin{center}\itshape\small
In celebration of the centenary of the first edition\\
of Watson's ``Treatise on the theory of Bessel functions''
\bigskip
\end{center}

\begin{abstract}
We sum in a close form the Sneddon-Bessel series
\[
   \sum_{m=1}^\infty
   \frac{J_\alpha(x j_{m,\nu})J_\beta(y j_{m,\nu})}
   {j_{m,\nu}^{2n+\alpha+\beta-2\nu+2} J_{\nu+1}(j_{m,\nu})^2},
\]
where $0<x$, $0<y$, $x+y<2$, $n$ is an integer, $\alpha,\beta,\nu\in \CC\setminus \{-1,-2,\dots \}$ with $2\Re \nu<2n+1+\Re \alpha+\Re \beta$ and $\{j_{m,\nu}\}_{m\geq 0}$ are the zeros of the Bessel function $J_\nu$ of order~$\nu$.
As an application we prove some extensions of the Kneser-Sommerfeld expansion.

\medskip


\medskip

\end{abstract}

\section{Introduction}

Sneddon considered in \cite[\S\,2.2]{Sne} the following Bessel series in two variables:
\begin{equation}
\label{lap}
   S_q^{\alpha,\beta,\nu}(x,y)
   = \sum_{m=1}^\infty
   \frac{J_\alpha(x j_{m,\nu})J_\beta(y j_{m,\nu})}{j_{m,\nu}^{q} J_{\nu+1}(j_{m,\nu})^2},
\end{equation}
where $0<x$, $0<y$, $x+y<2$ and $\{j_{m,\nu}\}_{m\geq 0}$ are the zeros of the Bessel function $J_\nu$ of order~$\nu$.

The purpose of this paper is to compute explicitly these Sneddon-Bessel series for
\[
   q_n = 2n+\alpha+\beta-2\nu+2,\quad n\in \ZZ,
\]
under mild conditions on the parameters $\alpha,\beta$ and $\nu$.

The problem of the explicit summation of Bessel series is a classical topic, but there is no doubt that it remains of interest today (\cite[\S\,6.8]{Br}, \cite{Gre}) and active research is still being done (specially in applied mathematics, mathematical physics and engineering).

In order to state our result in full detail, we need some notation. For $\nu \in \CC \setminus \{-1,-2,\dots \}$, let us consider the entire function
\begin{equation}
\label{def:Phibeta}
  \Phi_\nu(z) = 2^\nu \Gamma(\nu+1) \frac{J_\nu(z)}{z^\nu}
  = \Gamma(\nu+1) \sum_{n=0}^\infty \frac{(-1)^n (z/2)^{2n}}{n!\, \Gamma(n+\nu+1)}.
\end{equation}
Define now the polynomial $\delta_n^{\alpha,\beta,\nu}(x,y)$ (of degree~$2n$), $n\geq 0$, by the generating function
\[
   \frac{\Phi_{\alpha}(xz) \Phi_{\beta}(yz)}{\Phi_\nu(z)^2}
   = \sum_{n=0}^\infty \delta_n^{\alpha,\beta,\nu}(x,y)z^{2n}.
\]
This generating function allows the explicit computation of the polynomials $\delta_n^{\alpha,\beta,\nu}(x,y)$ recursively from the Taylor coefficients of the functions $\Phi_\alpha$, $\Phi_\beta$ and~$\Phi_\nu$.

We also define recursively the functions $\phi_n^{\alpha,\beta,\nu}(t)$, $n\geq 0$, by
\begin{align}
\label{hp1}
   \phi_0^{\alpha,\beta,\nu}(t)
   &= \frac{1}{\nu} \binom{\alpha}{\nu}\, \FF{\nu-\alpha,\nu}{\beta+1}{t^2}
   \\
\label{hp2}
   \phi_n^{\alpha,\beta,\nu}(t)
   &= \frac{1}{2\nu-2n}
   \left(
   \frac{1}{2(\alpha+1)} \phi_{n-1}^{\alpha+1,\beta,\nu}(t)
   + \frac{t^2}{2(\beta+1)} \phi_{n-1}^{\alpha,\beta+1,\nu}(t)
   \right),
\end{align}
where as usual $\FF{a,b}{c}{t}$ denotes the hypergeometric function
\[
   \FF{a,b}{c}{t} = \sum_{m=0}^{\infty}
   \frac{\Gamma(a+m) \Gamma(b+m) \Gamma(c)}{\Gamma(a) \Gamma(b) \Gamma(c+m)}
   \cdot \frac{t^m}{m!}.
\]
If we write
\[
   \delta_n^{\alpha,\beta,\nu}(x,y)
   = \sum_{2j + 2k \leq 2n} A_{2j,2k,n}^{\alpha,\beta,\nu} x^{2j} y^{2k},
\]
then for $\nu \notin \{0,1,\dots, n\}$, $2\Re\nu < 2n + 1 + \Re\alpha + \Re\beta$ and $0 < y \leq x < 1$ (and also for $x+y=2$ if $2\Re \nu <2n+\Re \alpha+\Re \beta$), we show that
\begin{multline}
\label{sod2i}
   S_{q_n}^{\alpha,\beta,\nu}(x,y)
   = \frac{\Gamma(\nu+1)^2 x^\alpha y^\beta}{2^{q_0}\Gamma(\alpha+1)\Gamma(\beta+1)}
   \bigg( x^{2n-2\nu} \phi_{n}^{\alpha,\beta,\nu}(y/x)
   \\
   - \sum_{2j + 2k \leq 2n}
   \frac{A_{2j,2k,n}^{\alpha,\beta,\nu}}{j+k+\nu - n} x^{2j} y^{2k}
   \bigg),
\end{multline}
if $n\geq 0$, and
\begin{equation}
\label{sod22}
   S_{q_n}^{\alpha,\beta,\nu}(x,y)
   = \frac{x^{\alpha-2\nu+2n}y^\beta\Gamma(\nu-n)}{2^{q_n}\Gamma(\beta+1)\Gamma(n+\alpha-\nu+1)}
   \, \FF{\nu-n,\nu-\alpha-n}{\beta+1}{\frac{y^2}{x^2}},
\end{equation}
if $n<0$. The case $\nu\in\{0,1,\dots, n\}$ can be computed by passing to the limit.

Some particular cases of \eqref{sod2i} and \eqref{sod22} are already known. For instance:
\begin{itemize}
\item[(a)] The case $x=1$, $\alpha=\mu+\nu+1$, $\beta=\nu$ and $n=-1$ is \cite[p.~690, (9)]{PBM}.

\item[(b)] The case $\alpha=\beta=\nu$, $n\geq 0$, $0\leq y\leq x\leq 1$, is packaged in the Kneser-Sommerfeld expansion (see \cite[(2)]{Mar})
\begin{equation}
\label{ks}
   \sum_{m=1}^\infty
   \frac{J_\nu(x j_{m,\nu})J_\nu(y j_{m,\nu})}{(j_{m,\nu}^{2}-z^2) J_{\nu+1}(j_{m,\nu})^2}
   = \frac{\pi J_\nu(yz)}{4J_\nu(z)}
   \big(Y_\nu(z)J_\nu(xz)-J_\nu(z)Y_\nu(xz)\big)
\end{equation}
(more precisely: $S_{q_n}^{\nu,\nu,\nu}(x,y)$, $n\geq 0$, are the Taylor coefficients at $z=0$ of the analytic function of $z$ on the right hand side).

\item[(c)] The case $n<0$ was computed by Sneddon \cite[\S\,2.2]{Sne} and more recently by Martin~\cite{Mar}.
\end{itemize}

The content of the paper is as follows. In Section~\ref{s1} we use the calculus of residues to find a partial fraction expansion of functions of the form $\frac{f(z)}{\Phi_\nu(z)^2}$, where $f$ is an entire function satisfying a suitable bound in $\CC$ (see Theorem \ref{teoremacuadrado} for details).

Theorem~\ref{teoremacuadrado} leads us to a partial differential equation for the Sneddon-Bessel series $S_{q_n}^{\alpha,\beta,\nu}(x,y)$, $n\geq 0$, which we solve in Section~\ref{osf} to find \eqref{sod2i}. Then, the identity~\eqref{sod22} can be easily deduced from the case $n=0$ by differentiation.

Using our identity \eqref{sod2i} for the Sneddon-Bessel series we find in Section~\ref{ksex} some extensions of the Kneser-Sommerfeld expansion \eqref{ks}, among which is the following:
\begin{equation}
\label{ksee}
   \sum_{m=1}^\infty
   \frac{j_{m,\nu}^{\nu-\beta}J_\nu(x j_{m,\nu})J_\beta(y j_{m,\nu})}
   {(j_{m,\nu}^{2}-z^2) J_{\nu+1}(j_{m,\nu})^2}
   = \frac{\pi J_\beta(yz)}{4 z^{\beta-\nu}J_\nu(z)}
   \big(Y_\nu(z)J_\nu(xz)-J_\nu(z)Y_\nu(xz)\big),
\end{equation}
if $\Re\nu < \Re\beta + 1$ and $0 \leq y \leq x \leq 1$ (as usual $Y_\nu$ denotes the Bessel function of the second kind).

\section{Preliminaries}

The zeros of the function $\Phi_\nu(w)$ defined by~\eqref{def:Phibeta}, that is, the zeros of the even function $J_{\nu}(w)/w^{\nu}$, are simple and can be ordered as a double sequence $\{j_{m,\nu}\}_{m\in \ZZ \setminus\{0\}}$ with $j_{-m,\nu} = -j_{m,\nu}$ and $0 \leq \Re j_{m,\nu} \leq \Re j_{m+1,\nu}$ for $m \geq 1$ (\cite[\S\,15.41, p.~497]{Wat}). Although these zeros depend on $\nu$, we will often omit this dependence to avoid unnecessary complications in the notation. The imaginary part of these zeros is bounded and, when $m$ is a sufficiently large integer, there is exactly one zero in the strip $m\pi + \frac{\pi}{2} \Re \nu + \frac{\pi}{4} < \Re z < (m+1)\pi + \frac{\pi}{2} \Re \nu + \frac{\pi}{4}$ (\cite[\S\,15.4, p.~497]{Wat}), so that
\[
   \lim_{m\to+\infty} \frac{|j_m|}{\pi m} = 1.
\]

It follows from the estimate
\[
   J_\nu(z)
   = 2^{1/2} (\pi z)^{-1/2}\left(
   \cos\Big(z - \frac{\nu}{2}\pi - \frac{\pi}{4}\Big) + o(1)
   \right),
   \quad z \to \infty
\]
(\cite[Eq.~10.7.8]{DLMF}, see also~\cite[\S\,7.21(1), p.~199]{Wat}) that
\[
   J_\nu(z)^2 + J_{\nu+1}(z)^2
   = \frac{2}{\pi z} \left(1 + e^{2 |\Im z|} o(1) \right),
   \quad z \to \infty,
\]
where the limit $z \to \infty$ is to be taken inside a sector $|\arg(z)| \leq \pi - \delta$. Thus,
\begin{equation}
\label{Jb+1(jm)}
   0 < c \leq |J_{\nu+1}(j_m)^2 j_m| \leq C
\end{equation}
for some constants $c$ and $C$ not depending on $m$. In terms of~$\Phi_{\nu+1}$,
\[
   0 < c \leq |\Phi_{\nu+1}(j_m)|^2 |j_m|^{2\Re\nu + 3} \leq C
\]
for some constants $c$ and $C$ not depending on~$m$.

Bessel functions satisfy the bound
\begin{equation}
\label{cotasupJ}
   |J_\nu(z)| \leq C \frac{e^{|\Im z|}}{|z|^{1/2}},
\end{equation}
for $|z|$ large enough, with a constant $C$ depending only on~$\nu$.
To be precise, for $|z| > \varepsilon > 0$ and $\nu$ on a compact set $K$, there is a constant $C$ depending only on $\varepsilon$ and $K$, as follows from~\cite[Eq.~10.4.4 and \S\,10.17(iv)]{DLMF}.

For $\mu$ and $\eta$ satisfying $\Re \mu>\Re \eta>-1$, consider the integral transform $T_{\mu,\eta}$ given by
\begin{equation}
\label{itc}
   T_{\mu,\eta}(f)(x) = \frac{1}{2^{\mu-\eta-1}\Gamma(\mu-\eta)}
   \int_0^1 f(xs) s^{2\eta+1} (1-s^2)^{\mu-\eta-1} \, ds
\end{equation}
(with a small abuse of notation, we will often write $T_{\mu,\eta}(f(x))$ if it does not cause confusion).

Sonin's formula for the Bessel functions (\cite[12.11(1), p.~373]{Wat}) can be written as
\begin{equation}
\label{eq:sonin}
   \frac{J_{\mu}(x)}{x^\mu} = \frac{2^{\eta+1-\mu}}{\Gamma(\mu-\eta)}
   \int_0^1 \frac{J_\eta(xs)}{(xs)^{\eta}} s^{2\eta+1} (1-s^2)^{\mu-\eta-1} \, ds
   = T_{\mu,\eta} \left(\frac{J_\eta(x)}{x^{\eta}}\right),
\end{equation}
valid for $\Re \mu > \Re \eta>-1$.
For $2\Re \eta+r+2>0$, we also have
\begin{equation}
\label{eaq}
   T_{\mu,\eta}(x^{r}) =
   \frac{\Gamma(\eta+\frac{r}{2}+1)}{2^{\mu-\eta} \Gamma(\mu+\frac{r}{2}+1)} \,x^{r},
\end{equation}
as follows from the identity
\[
   \int_0^1 s^{a}(1-s^2)^{b}\,ds
   = \frac{\Gamma(\frac{a+1}{2}) \Gamma(b+1)}{2 \Gamma(\frac{a+1}{2} + b + 1)},
   \qquad \Re a,\ \Re b >-1.
\]
The identities \eqref{eq:sonin} and \eqref{eaq} can be extended for $\Re \eta < -1$ as follows. For complex numbers $\mu$, $\eta$ and a positive integer $h$ satisfying $\Re \eta > - \frac{h}{2} - 1$, $\Re \mu>\Re \eta+h$, consider the integral transform $T_{\mu,\eta,h}$ given by
\begin{equation}
\label{itch}
   T_{\mu,\eta,h}(f)(x) = \frac{(-1)^h 2^{\eta+1-\mu}\Gamma(2\eta+2)}{\Gamma(\mu-\eta) \Gamma(2\eta+2+h)}
    \int_0^1 \frac{d^h}{ds^h}(f(xs)(1-s^2)^{\mu-\eta-1})
   s^{2\eta+h+1} \,ds.
\end{equation}
To be precise, $\eta$ should not be half a negative integer (in case it is we will manage somehow).

It is then easy to check that
\begin{align}
\notag
   T_{\mu,\eta,h}(x^{r})
   &= \frac{\Gamma(\eta+\frac{r}{2}+1)}{2^{\mu-\eta}\Gamma(\mu+\frac{r}{2}+1)} \,x^{r},
   \\
\label{lab2}
   T_{\mu,\eta,h} \left(\frac{J_\eta(x)}{x^{\eta}}\right)
   &= \frac{J_{\mu}(x)}{x^\mu}.
\end{align}

\section{Partial fraction decomposition of Bessel functions}
\label{s1}

In this section, we use the calculus of residues to find a partial fraction expansion of functions of the form $\frac{f(z)}{\Phi_\nu(z)^2}$, where $f$ is an entire function with some growth control.

\begin{theorem}
\label{teoremacuadrado}
Let $f$ be an entire function satisfying
\[
   |f(z)| \leq c (1+|z|)^N e^{\kappa|\Im z|},
   \quad z \in \CC,
\]
for certain constants $c > 0$, $N \in \RR$, and $\kappa \leq 2$. Let $\nu \in \CC \setminus \{-1,-2,\dots\}$ and $n$ be a nonnegative integer such that
\[
   N + 1 + 2 \Re\nu < n,
   \quad \text{if }\, \kappa = 2,
\]
or
\[
   N + 2 \Re\nu < n,
   \quad \text{if }\, \kappa < 2.
\]
Then
\begin{multline*}
   \frac{1}{n!} \frac{d^n}{dt^n} \left( \frac{f(t)}{\Phi_\nu(t)^2} \right)
   \\
   = \sum_{m \in \ZZ \setminus \{0\}} 4(\nu+1)^2
   \frac{\big((2\nu+1) t - (2\nu-n) j_m\big) f(j_m) - j_m(j_m-t) f'(j_m)}
   {(j_m-t)^{n+2} j_m^3 \Phi_{\nu+1}(j_m)^2},
\end{multline*}
where the convergence is uniform in bounded subsets of $\CC \setminus \{j_m : m \in \ZZ \setminus \{0\} \}$.
\end{theorem}

\begin{proof}
Let us fix $t \in \CC \setminus \{j_m : m \in \ZZ \setminus \{0\} \}$ and consider the holomorphic function $\frac{f(w)}{(w-t)^{n+1} \Phi_\nu(w)^2}$. It has a pole at $t$ of order $n+1$, and a double pole at each $j_m$, $m\in \ZZ \setminus \{0\}$. The residue at $t$ is, therefore,
\[
   \frac{1}{n!} \frac{d^n}{dw^n} \left( \frac{f(w)}{\Phi_\nu(w)^2} \right)_{w=t},
\]
while the residue at each $j_m$ is
\begin{multline}
\label{residuo}
   \lim_{w\to j_m} \frac{d}{dw} \left( \frac{(w-j_m)^2 f(w)}{(w-t)^{n+1} \Phi_\nu(w)^2} \right)
   \\
   = \frac{d}{dw} \left(\frac{f(w)}{(w-t)^{n+1}}\right)_{w=j_m}
   \left( \lim_{w\to j_m} \frac{w-j_m}{\Phi_\nu(w)} \right)^2
   \\
   + \frac{f(j_m)}{(j_m-t)^{n+1}} \frac{d}{dw} \left(
   \left(\frac{w-j_m}{\Phi_\nu(w)} \right)^2
   \right)_{w=j_m}.
\end{multline}
Let us consider separately the last term:
\begin{align*}
   \frac{d}{dw}\left(
   \left(\frac{w-j_m}{\Phi_\nu(w)} \right)^2
   \right)_{w=j_m}
   &= 2 \frac{1}{\Phi_\nu'(j_m)}
   \lim_{w \to j_m} \frac{\Phi_\nu(w) - (w-j_m) \Phi_\nu'(w)}{\Phi_\nu(w)^2}
   \\
   &= \frac{2}{\Phi_\nu'(j_m)}
   \lim_{w \to j_m} \frac{-(w-j_m) \Phi_\nu''(w)}{2 \Phi_\nu(w) \Phi_\nu'(w)}
   \\
   &= \frac{2}{\Phi_\nu'(j_m)} \cdot \frac{-\Phi_\nu''(j_m)}{2\Phi_\nu'(j_m) \Phi_\nu'(j_m)}
   \\
   &= -\frac{\Phi_\nu''(j_m)}{\Phi_\nu'(j_m)^3}.
\end{align*}
Now, the identities
\begin{align}
\label{Phiprima}
   \Phi_\nu'(z) &= -\frac{z}{2(\nu+1)} \Phi_{\nu+1}(z),
   \\
\notag
   \Phi_\nu''(z) &= - \Phi_\nu(z) + \frac{2\nu+1}{2(\nu+1)} \Phi_{\nu+1}(z)
\end{align}
(see~\cite[\S\,3.2, p.~45]{Wat}) prove that
\[
   -\frac{\Phi_\nu''(j_m)}{\Phi_\nu'(j_m)^3}
   = \frac{4 (\nu+1)^2 (2\nu+1)}{j_m^3 \Phi_{\nu+1}(j_m)^2},
\]
so that, going back to~\eqref{residuo} and using~\eqref{Phiprima} again, the residue at $j_m$ is
\begin{gather*}
   \frac{(j_m - t) f'(j_m) - (n+1) f(j_m)}{(j_m-t)^{n+2}}
   \Big( \frac{1}{\Phi_\nu'(j_m)} \Big)^2
   + \frac{f(j_m)}{(j_m-t)^{n+1}}
   \cdot \frac{4 (\nu+1)^2 (2\nu+1)}{j_m^3 \Phi_{\nu+1}(j_m)^2}
   \\
   \begin{split}= \frac{(j_m - t) f'(j_m) - (n+1) f(j_m)}{(j_m-t)^{n+2}}
   \cdot \frac{4(\nu+1)^2}{j_m^2 \Phi_{\nu+1}(j_m)^2}
   \\ \null+ \frac{f(j_m)}{(j_m-t)^{n+1}}
   \cdot \frac{4 (\nu+1)^2 (2\nu+1)}{j_m^3 \Phi_{\nu+1}(j_m)^2}\end{split}
   \\
   = 4(\nu+1)^2 \, \frac{j_m (j_m-t) f'(j_m) + \big((2\nu-n) j_m - (2\nu+1) t\big) f(j_m)}
   {(j_m-t)^{n+2} j_m^3 \Phi_{\nu+1}(j_m)^2}.
\end{gather*}
Thus, if $D = \{z \in \CC : |z| = A\}$ is a large circle of radius $A > |t|$ with the only condition, at the moment, that none of the points $j_m$ lie in $D$, the calculus of residues gives
\begin{multline}
\label{integralporresiduos}
   \frac{1}{2\pi i} \int_D \frac{f(w)}{(w-t)^{n+1} \Phi_\nu(w)^2} \, dw
   = \frac{1}{n!} \frac{d^n}{dw^n} \left( \frac{f(w)}{\Phi_\nu(w)^2} \right)_{w=t}
   \\
   + \sum_{|j_m| < A} 4(\nu+1)^2 \,
   \frac{j_m (j_m-t) f'(j_m) + \big((2\nu-n) j_m - (2\nu+1) t\big) f(j_m)}
   {(j_m-t)^{n+2} j_m^3 \Phi_{\nu+1}(j_m)^2}.
\end{multline}
Now, the value of $A$ can be chosen arbitrarily large and such that there exists some constant $c>0$, independent of~$A$, satisfying
\[
   c\, \frac{e^{|\Im w|}}{|w|^{1/2}} \leq |J_\nu(w)|
\]
for $w \in D$ (see~\cite[formula (2.4)]{DPV}). Thus,
\[
   \bigg| \frac{f(w)}{(w-t)^{n+1} \Phi_\nu(w)^2} \bigg|
   = C\, \frac{|f(w)| |w^\nu|^2}{|w-t|^{n+1} |J_\nu(w)|^2}
   \leq C\, \frac{A^{N+1 + 2\Re\nu} e^{(\kappa-2)|\Im w|}}{(A - |t|)^{n+1}}
\]
for $w \in D$, where $C$ is a constant, independent of $A$, but possibly different at each occurrence. The natural parametrization of $D$ then gives
\begin{align*}
   \bigg| \frac{1}{2\pi i} \int_D \frac{f(w)}{(w-t)^{n+1} \Phi_\nu(w)^2} \, dw \bigg|
   &\leq \frac{C}{2\pi} \int_{-\pi}^\pi
   \frac{A^{N+2 + 2\Re\nu} e^{(\kappa-2)A|\sin s|}}{(A - |t|)^{n+1}} \, ds
   \\
   &= \frac{2C}{\pi} \frac{A^{N+2 + 2\Re\nu}}{(A - |t|)^{n+1}}
   \int_0^{\pi/2} e^{(\kappa-2)A|\sin s|} \, ds
   \\
   &\leq \frac{2C}{\pi} \frac{A^{N+2 + 2\Re\nu}}{(A - |t|)^{n+1}}
   \int_0^{\pi/2} e^{(\kappa-2)A 2s/\pi} \, ds.
\end{align*}
Now, the last integral is obviously a constant if $\kappa = 2$, while it is $O(A^{-1})$ if $\kappa < 2$.
Taking this bound into~\eqref{integralporresiduos} and letting $A$ be arbitrarily large proves the theorem.
\end{proof}

Evaluating at $t=0$ the identity of Theorem~\ref{teoremacuadrado}, gives
\begin{equation}
\label{evalt=0}
   \frac{1}{n!}
   \frac{d^n}{dt^n} \left( \frac{f(t)}{\Phi_\nu(t)^2} \right)_{t=0}
   = - 4(\nu+1)^2 \sum_{m \in \ZZ \setminus \{0\}}
   \frac{(2\nu-n) f(j_m) + j_m f'(j_m)}{j_m^{n+4} \Phi_{\nu+1}(j_m)^2},
\end{equation}
under the assumption that
\begin{align*}
   N + 1 + 2\Re \nu &< n, \quad \text{if } \kappa = 2,
   \\
   N + 2\Re \nu &< n, \quad \text{if } \kappa < 2.
\end{align*}
We then define the double Bessel numbers $\delta_n^{f,\nu}$ by
\[
   \delta_n^{f,\nu}
   = \frac{1}{n!}
   \frac{d^n}{dt^n} \left( \frac{f(t)}{\Phi_\nu(t)^2} \right)_{t=0}.
\]
These are the Taylor coefficients of $\frac{f(t)}{\Phi_\nu(t)^2}$ at $t=0$, in other words,
\[
   \frac{f(t)}{\Phi_\nu(t)^2} = \sum_{n=0}^\infty \delta_n^{f,\nu}t^n
\]
in a neighbourhood of~$0$.

\section{Summing Sneddon-Bessel series explicitly}
\label{osf}

Our goal is to sum the Sneddon-Bessel series
\[
   \sum_{m\geq 1}
   \frac{J_{\alpha} (xj_m) J_{\beta}(yj_m)}
   {j_m^{2n+\alpha+\beta-2\nu+2} J_{\nu +1}(j_m)^2},
\]
where $\alpha, \beta,\nu \in \CC \setminus \{-1,-2,\dots\}$, $0 < x$, $0 < y$, $x+y\leq 2$, and $n$ is an integer.
To this end, let us take
\begin{equation}
\label{defxixy}
   \xi_{n,\alpha,\beta,\nu}(x,y)
   = \sum_{m\geq 1}
   \frac{\Phi_{\alpha}(xj_m) \Phi_{\beta}(yj_m)}{j_m^{2n+4} \Phi_{\nu +1}(j_m)^2},
\end{equation}
with the condition that
\begin{equation}
\label{rangooptimo}
   2 \Re \nu < 2n+1 + \Re \alpha + \Re \beta.
\end{equation}
According to~\eqref{Jb+1(jm)} and~\eqref{cotasupJ}, this guarantees that the series converges absolutely. These series are related to the Sneddon-Bessel series \eqref{lap} by
\begin{equation}
\label{lap2}
   S_{q_n}^{\alpha,\beta,\nu}(x,y)
   = \frac{\Gamma(\nu+2)^2 x^\alpha y^\beta}
   {2^{\alpha +\beta - 2\nu - 2} \Gamma(\alpha+1)\Gamma(\beta+1)}
   \xi_{n,\alpha,\beta,\nu}(x,y).
\end{equation}

Under the stronger condition
\begin{equation}
\label{rangopeq2v}
   2\Re \nu < 2n + \Re \alpha + \Re \beta,
\end{equation}
termwise differentation in~\eqref{defxixy} is allowed. In particular, we obtain
\begin{equation}
\label{ec:parcialxi}
   \frac{\partial}{\partial x} \xi_{n,\alpha,\beta,\nu}(x,y)
   = - \frac{x}{2(\alpha+1)} \xi_{n-1,\alpha+1,\beta,\nu}(x,y)
\end{equation}
(and the same for the other partial derivative).

\subsection{The case $n\geq 0$}

Let us assume firstly that $n$ is a nonnegative integer (later on we will address the case when $n$ is negative).

The function $f(z) = \Phi_{\alpha}(xz) \Phi_{\beta}(yz)$ meets the conditions of Theorem~\ref{teoremacuadrado} with $N=-\Re \alpha - \Re \beta -1$ and $\kappa = x + y$, and the condition $N + 2 \Re \nu < 2n$ of Theorem~\ref{teoremacuadrado} is therefore \eqref{rangooptimo}. Thus, \eqref{evalt=0} becomes
\begin{equation}
\label{eddp}
   \frac{\delta_n^{\alpha,\beta,\nu}(x,y)}{-8(\nu+1)^2}
   = (2\nu-2n) \xi_{n,\alpha,\beta,\nu}(x,y)
   + x \frac{\partial \xi_{n,\alpha,\beta,\nu} (x,y)}{\partial x}
   + y \frac{\partial \xi_{n,\alpha,\beta,\nu} (x,y)}{\partial y},
\end{equation}
where the function
\begin{equation}
\label{eb1}
   \delta_n^{\alpha,\beta,\nu}(x,y) = \frac{1}{(2n)!}
   \frac{d^{2n}}{dz^{2n}}
   \left( \frac{\Phi_{\alpha}(xz)\Phi_{\beta}(yz)}{\Phi_\nu(z)^2} \right)_{z=0}
\end{equation}
is a polynomial in $x^2$ and $y^2$ (that is, even powers of $x$ and $y$) which could be computed recursively from the Taylor coefficients of the functions $\Phi_{\alpha}$, $\Phi_\beta$ and $\Phi_\nu$ involved. Notice that
\begin{equation}
\label{jodr}
   \frac{\Phi_\alpha(xz)\Phi_\beta(yz)}{\Phi_\nu(z)^2}
   = \sum_{n=0}^\infty \delta_n^{\alpha,\beta,\nu}(x,y)z^{2n}.
\end{equation}

Let us write
\begin{equation}
\label{def:Ajk}
   \delta_n^{\alpha,\beta,\nu}(x,y) = \sum_{2j + 2k \leq 2n} A_{2j,2k,n}^{\alpha,\beta,\nu} x^{2j} y^{2k},
\end{equation}
and assume also, for simplicity, that $\nu\neq 0,1,\dots, n$.
Then, it is easy to see that the solution to~\eqref{eddp} is
\begin{multline}
\label{sod2}
   \xi_{n,\alpha,\beta,\nu}(x,y)
   = \frac{1}{16(\nu+1)^2} \bigg(
   {-} \sum_{2j + 2k \leq 2n}
   \frac{A_{2j,2k,n}^{\alpha,\beta,\nu}}{j+k+\nu - n} \,x^{2j} y^{2k}
   \\
   + x^{2n-2\nu} \phi_{n,\alpha,\beta,\nu}(y/x)
   \bigg),
\end{multline}
if~\eqref{rangopeq2v} holds, where $\phi_{n,\alpha,\beta,\nu}$ is a one variable function to be determined. In case $\nu \in \{0,1,\dots,n\}$, some logarithmic terms appear also.

Before going on, let us focus on the dependence of these functions and constants on the parameter $\alpha$ and $\beta$. It is apparent from~\eqref{eb1} and~\eqref{def:Ajk} that each $A_{2j,2k,n}^{\alpha,\beta,\nu}$ is a rational function of $\alpha$, $\beta$ and $\nu$.
If $n$, $\nu$, $x$, $y$ and $\beta$ (respectively, $\alpha$) are fixed, then the function $\Phi_\alpha(x j_m)$ is holomorphic on $\alpha \in \CC \setminus \{-1,-2,\dots\}$, and so is $\xi_{n,\alpha,\beta,\nu}(x,y)$ (resp., $\beta$) under the condition~\eqref{rangooptimo} (the series involved converge uniformly on $\alpha$-compacts, as follows from~\eqref{def:Phibeta} and~\eqref{cotasupJ}). The same applies therefore to $\phi_{n,\alpha,\beta,\nu}(y/x)$. This analytic dependence on~$\alpha$ (resp.,~$\beta$) will eventually allow us to extend some identities by analytic continuation.

Thus, formula~\eqref{sod2}, which in principle requires~\eqref{rangopeq2v} to hold, extends to the whole range~\eqref{rangooptimo} in this way: firstly, \eqref{ec:parcialxi} can be written as
\[
   \xi_{n,\alpha,\beta,\nu}(x,y)
   = - \frac{2\alpha}{x} \frac{\partial}{\partial x} \xi_{n+1,\alpha-1,\beta,\nu}(x,y)
\]
on the whole range~\eqref{rangooptimo}; using now~\eqref{sod2} on the right-hand side gives an expression for $\xi_{n,\alpha,\beta,\nu}(x,y)$ with holomorphic coefficients, which by analytic continuation must equal the coefficients in~\eqref{sod2}.

In view of~\eqref{sod2}, it is enough to find the function $\phi_{n,\alpha,\beta,\nu}$ to explicitly determine the function $\xi_{n,\alpha,\beta,\nu}$.
So let us now find a recursion for the functions $\phi_{n,\alpha,\beta,\nu}$. Given $0 < t < 1$, let us write
\[
   \varphi_{n,\alpha,\beta,\nu}(s)=\xi_{n,\alpha,\beta,\nu}(s,ts)
\]
for $s$ small enough.
Then~\eqref{ec:parcialxi} yields
\begin{align*}
   \varphi'_{n,\alpha,\beta,\nu}(s)
   &= \frac{\partial}{\partial x} \xi_{n,\alpha,\beta,\nu}(s,ts)
   + t \frac{\partial}{\partial y} \xi_{n,\alpha,\beta,\nu}(s,ts)
   \\
   &= - \frac{s}{2(\alpha+1)} \xi_{n-1,\alpha+1,\beta,\nu}(s,ts)
   - \frac{s t^2}{2(\beta+1)} \xi_{n-1,\alpha,\beta+1,\nu}(s,ts).
\end{align*}
The coefficient of $s^{2n-1-2\nu}$ on the right-hand side, as follows from~\eqref{sod2}, is
\[
   \frac{1}{16(\nu+1)^2} \left( - \frac{1}{2(\alpha+1)} \phi_{n-1,\alpha+1,\beta,\nu}(t)
   - \frac{t^2}{2(\beta+1)} \phi_{n-1,\alpha,\beta+1,\nu}(t) \right).
\]
On the other hand, \eqref{sod2} translates into
\begin{multline*}
   \varphi_{n,\alpha,\beta,\nu}(s)
   = \frac{1}{16(\nu+1)^2} \bigg(
   {-} \sum_{2j + 2k \leq 2n}
   \frac{A_{2j,2k,n}^{\alpha,\beta,\nu}}{j+k+\nu - n} \,s^{2j+2k} t^{2k}
   \\
   + s^{2n-2\nu} \phi_{n,\alpha,\beta,\nu}(t)
   \bigg),
\end{multline*}
so that the coefficient of $s^{2n-1-2\nu}$ in $\varphi'_{n,\alpha,\beta,\nu}(s)$ is
\[
   \frac{2n-2\nu}{16(\nu+1)^2} \phi_{n,\alpha,\beta,\nu}(t).
\]
Equating both formulas for the coefficient of $s^{2n-1-2\nu}$ results in
\begin{equation}
\label{rdc}
   \phi_{n,\alpha,\beta,\nu}(t) = \frac{1}{2\nu-2n} \left(
   \frac{1}{2(\alpha+1)} \phi_{n-1,\alpha+1,\beta,\nu}(t)
   + \frac{t^2}{2(\beta+1)} \phi_{n-1,\alpha,\beta+1,\nu}(t)
   \right).
\end{equation}
This recursion reduces the problem of finding $\xi_{n,\alpha,\beta,\nu}(x,y)$ to the case $n=0$, so let us concentrate on this. We first consider the case $\alpha = \beta = \nu$, then address the general case. Observe that condition~\eqref{rangooptimo} holds for $n=0$, $\alpha = \beta = \nu$. Now, \eqref{eb1} gives $\delta_0^{\alpha,\beta,\nu}(x,y) = 1$, so that \eqref{sod2} is
\[
   \xi_{0,\nu,\nu,\nu}(x,y)
   = \frac{1}{16(\nu+1)^2} \bigg(
   {-} \frac{1}{\nu} + x^{-2\nu} \phi_{0,\nu,\nu,\nu}(y/x)
   \bigg).
\]
Since $\Phi_\nu(j_m) = 0$, the definition~\eqref{defxixy} trivially gives
\[
   \xi_{0,\nu,\nu,\nu}(1,t) = 0,
   \quad 0 < t < 1.
\]
Therefore, $\phi_{0,\nu,\nu,\nu}(t) = \frac{1}{\nu}$ for $0 < t < 1$ and (by symmetry)
\begin{equation}
\label{cop}
   \xi_{0,\nu,\nu,\nu}(x,y) 
   = \begin{cases}
   \frac{1}{16\nu (\nu+1)^2}
   \left( -1 + x^{-2\nu} \right), & \text{if } y \leq x,
   \\
   \frac{1}{16\nu (\nu+1)^2}
   \left( -1 + y^{-2\nu} \right),& \text{if } x < y.
   \end{cases}
\end{equation}
Let us now find $\xi_{0,\alpha,\beta,\nu}(x,y)$. By symmetry, we can assume that $y\leq x$ without loss of generality. Sonin's formula~\eqref{eq:sonin} easily gives
\begin{multline*}
   \xi_{0,\alpha,\beta,\nu}(x,y)
   = \frac{4\Gamma(\alpha+1) \Gamma(\beta+1)}
   {\Gamma(\nu+1)^2\Gamma(\alpha-\nu)\Gamma(\beta-\nu)}
   \\
   \quad \times \iint_{[0,1]\times[0,1]}
   \xi_{0,\nu,\nu,\nu}(xr, ys) r^{2\nu+1} s^{2\nu+1}
   (1-r^2)^{\alpha-\nu-1} (1-s^2)^{\beta-\nu-1} \, dr \, ds
\end{multline*}
with the additional restrictions $\Re \alpha > \Re \nu > -1$, $\Re \beta > \Re \nu > -1$.
To evaluate this integral, let us separate the square $[0,1]\times [0,1]$ into the sets
\begin{align*}
   A_{x,y} &= \{(r,s) : x r \geq y s \},
   \\
   B_{x,y} &= \{(r,s) : x r < y s \},
\end{align*}
so that
\[
   \xi_{0,\nu,\nu,\nu}(x \sqrt{r},y \sqrt{s})
   = \begin{cases}
   \frac{1}{16\nu(1+\nu)^2} \left(-1+x^{-2\nu} r^{-2\nu} \right),
   & (r,s) \in A_{x,y},
   \\
   \frac{1}{16\nu(1+\nu)^2} \left(-1+y^{-2\nu}s^{-2\nu}\right),
   & (r,s) \in B_{x,y}.
   \end{cases}
\]
The above integral is therefore
\begin{align*}
   \frac{1}{16\nu(\nu+1)^2}
   \Bigg(
   &{-}\int_0^1 \int_0^1 r^{2\nu+1} s^{2\nu+1} 
   (1-r^2)^{\alpha-\nu-1} (1-s^2)^{\beta-\nu-1} \,dr \, ds
   \\
   &+ \iint_{A_{x,y}} x^{-2\nu} r^{-2\nu} r^{2\nu+1} s^{2\nu+1} (1-r^2)^{\alpha-\nu-1}
   (1-s^2)^{\beta-\nu-1} \, dr \, ds
   \\
   &+ \iint_{B_{x,y}} y^{-2\nu} s^{-2\nu} r^{2\nu+1} s^{2\nu+1} (1-r^2)^{\alpha-\nu-1}
   (1-s^2)^{\beta-\nu-1} \, dr \, ds
   \Bigg).
\end{align*}
The first of these three integrals is immediate:
\[
   \int_0^1 \int_0^1 r^{2\nu+1} s^{2\nu+1} (1-r^2)^{\alpha-\nu-1} (1-s^2)^{\beta-\nu-1} \, dr \, ds
   = \frac{\Gamma(\alpha-\nu) \Gamma(\beta-\nu) \Gamma(\nu+1)^2}
   {4\Gamma(\alpha+1) \Gamma(\beta+1)}.
\]
Taking into account that $y \leq x$, the second integral is
\begin{align*}
   &\iint_{A_{x,y}} x^{-2\nu} r^{-2\nu} r^{2\nu+1} s^{2\nu+1} (1-r^2)^{\alpha-\nu-1}
   (1-s^2)^{\beta-\nu-1} \, dr \, ds
   \\
   &\qquad= x^{-2\nu} \int_0^1
   \bigg(\int_{y^2 s^2/x^2}^1 r(1-r^2)^{\alpha-\nu-1} \, dr\bigg)
   s^{2\nu+1} (1-s^2)^{\beta-\nu-1} \, ds
   \\
   &\qquad= \frac{x^{-2\nu}}{2(\alpha-\nu)}
   \int_0^1 \left(1-\frac{y^2}{x^2}s^2 \right)^{\alpha-\nu}
   s^{2\nu+1} (1-s^2)^{\beta-\nu-1} \, ds
   \\
   &\qquad= \frac{x^{-2\nu} \Gamma(\nu+1) \Gamma(\beta-\nu)}{4(\alpha-\nu)\Gamma(\beta+1)}
   \,\FF{\nu-\alpha, \nu+1}{\beta+1}{\frac{y^2}{x^2}},
\end{align*}
where the integral representation of the hypergeometric function~${}_{2}F_{1}$ is used.

The third integral is
\begin{align*}
   &\iint_{B_{x,y}} y^{-2\nu} s^{-2\nu} r^{2\nu+1} s^{2\nu+1} (1-r^2)^{\alpha-\nu-1}
   (1-s^2)^{\beta-\nu-1} \, dr \, ds
   \\
   &\qquad= y^{-2\nu} \int_0^{y^2/x^2}
   \bigg( \int_{x^2 r^2 2/y^2}^1 s(1-s^2)^{\beta-\nu-1} \, ds\bigg)
   r^{2\nu+1} (1-r^2)^{\alpha-\nu-1} \,dr
   \\
   &\qquad= \frac{y^{-2\nu}}{2(\beta-\nu)}
   \int_0^{y^2/x^2} \left(1-\frac{x^2}{y^2} r^2 \right)^{\beta-\nu}
   r^{2\nu+1} (1-r^2)^{\alpha-\nu-1} \, dr
   \\
   &\qquad= \frac{y^2 x^{-2\nu-2}}{2(\beta-\nu)}
   \int_0^1 (1-t^2)^{\beta-\nu} t^{2\nu+1} \left(1-\frac{y^2}{x^2} t^2\right)^{\alpha-\nu-1}
   \, dt
   \\
   &\qquad= \frac{y^2 x^{-2\nu-2} \Gamma(\nu+1) \Gamma(\beta-\nu+1)}
   {4(\beta-\nu) \Gamma(\beta+2)}
   \, \FF{\nu-\alpha+1,\nu+1}{\beta+2}{\frac{y^2}{x^2}}.
\end{align*}
Putting together all the pieces,
\begin{multline*}
   \xi_{0,\alpha,\beta,\nu}(x,y)
   = \frac{1}{16\nu (\nu+1)^2} \Bigg(
   {-1}
   + \frac{x^{-2\nu} \Gamma(\alpha+1)}{\Gamma(\nu+1) \Gamma(\alpha-\nu+1)}
   \\
   \times \bigg(\FF{\nu-\alpha,\nu+1}{\beta+1}{\frac{y^2}{x^2}}
   + \frac{y^2 (\alpha-\nu)}{x^2 (\beta+1)}
   \, \FF{\nu-\alpha+1,\nu+1}{\beta+2}{\frac{y^2}{x^2}}
   \bigg)
   \Bigg).
\end{multline*}
Finally, the elementary relation
\[
  \FF{a,b+1}{c}{t} - t \,\frac{a}{c} \, \FF{a+1,b+1}{c+1}{t}
   = \FF{a,b}{c}{t}
\]
gives
\begin{align}
\label{forn2}
   \xi_{0,\alpha,\beta,\nu}(x,y) &= \frac{1}{16\nu(\nu+1)^2}\left(
   -1
   + x^{-2\nu} \binom{\alpha}{\nu}
   \,
   \FF{\nu-\alpha,\nu}{\beta+1}{\frac{y^2}{x^2}}
   \right),
\end{align}
valid for $\Re \alpha, \Re \beta > \Re \nu > -1$, and $0 < y \leq x$, $x+y < 2$.

Assuming that $\Re \nu > -1$, the identity \eqref{forn2} extends to the whole range given by~\eqref{rangooptimo}, i.e., $2\Re \nu <1+\Re \alpha+\Re \beta$, by an argument of analyticity (and also for $x+y=2$ if $2\Re \nu <\Re \alpha+\Re \beta$).

Let us consider now the case $\Re \nu<-1$. Take a positive integer $h$ such that $\Re \nu >-h/2-1$ and $\alpha,\beta $ satisfying $\Re\alpha, \Re \beta >\Re \nu+h$. Let us assume for the moment that $\nu$ is not half a negative integer; using the integral transform $T_{\alpha,\nu,h}$ defined by~\eqref{itch} acting on $x$ and $T_{\beta,\nu,h}$ acting on $y$, we get from~\eqref{lab2}:
\begin{gather*}
   \xi_{0,\alpha,\beta,\nu}(x,y)
   = \frac{4\Gamma(\alpha+1)\Gamma(\beta+1)}
   {\Gamma(\nu+1)^2\Gamma(\alpha-\nu)\Gamma(\beta-\nu)}
   \cdot \frac{\Gamma(2\nu+2)^2}{\Gamma(2\nu+2+h)^2}
   \\
   \times \iint\limits_{[0,1]\times[0,1]}
   \frac{\partial ^{2h}}{\partial r^h\,\partial s^h}
   \Big(\xi_{0,\nu,\nu,\nu}(x r, ys)(1-r^2)^{\alpha-\nu-1}(1-s^2)^{\beta-\nu-1}\Big) 
   (rs)^{2\nu +h+1} \, dr \, ds.
\end{gather*}
The function $\xi_{0,\nu,\nu,\nu}(x, y)$ is given by \eqref{cop}. Therefore, $\xi_{0,\nu,\nu,\nu}$ is analytic in $\nu$ and so is the function on the right hand side of the above identity (on the region $\Re \nu >-h/2-1$). For $\nu>-1$, integrating by parts we deduce that this function is equal to
\begin{multline}
\label{coj}
   \frac{4\Gamma(\alpha+1)\Gamma(\beta+1)}{\Gamma(\nu+1)^2\Gamma(\alpha-\nu)\Gamma(\beta-\nu)}
   \\
   \times \iint_{[0,1]\times[0,1]}
   \xi_{0,\nu,\nu,\nu}(xr, ys) (1-r^2)^{\alpha-\nu-1} (1-s^2)^{\beta-\nu-1} (rs)^{2\nu +1}
   \, dr \, ds.
\end{multline}
Proceeding as before, we deduce that for $\nu>-1$ and $0<y\leq x<1$, the function \eqref{coj} is equal to the function on the right hand side of \eqref{forn2}, which is also analytic in $\nu$. This proves the identity \eqref{forn2} also for $\Re \nu >-h/2-1$ and $\Re\alpha, \Re \beta>\Re \nu+h$. Using again an argument of analyticity on the variables $\alpha$ and $\beta$, we prove that \eqref{forn2} holds indeed for $2\Re \nu <1+\Re \alpha+\Re \beta$ (and also for $x+y=2$ if $2\Re \nu <\Re \alpha+\Re \beta$). The requirement that $\nu$ is not half a negative integer can be suppressed by continuity.

By the way, this means that
\begin{equation}
\label{phi0}
   \phi_{0,\alpha,\beta,\nu}(t)
   = \frac{1}{\nu} \binom{\alpha}{\nu}
   \,\FF{\nu-\alpha,\nu}{\beta+1}{t^2},
   \quad 0<t<1,
\end{equation}
which, together with~\eqref{rdc}, allows to find the functions $\phi_{n,\alpha,\beta,\nu}$ and, therefore, $\xi_{n,\alpha,\beta,\nu}$ for every positive integer~$n$.

For the sake of completeness, we display in full extension the cases $n=0,1$.

\subsubsection{The case $n=0$}

Identities \eqref{lap2} and \eqref{forn2} give
\[
   S_{q_0}^{\alpha,\beta,\nu}(x,y)
   = \frac{\Gamma(\nu+1)^2 x^\alpha y^\beta}{2^{q_0}\nu \Gamma(\alpha+1)\Gamma(\beta+1)}
   \left(
   -1 +  x^{-2\nu} \binom{\alpha}{\nu} \, \FF{\nu-\alpha,\nu}{\beta+1}{\frac{y^2}{x^2}}
   \right),
\]
valid for $2\Re\nu < 1 + \Re\alpha + \Re\beta$, $\nu\neq 0$, and $0 < y \leq x$, $x+y < 2$ (and also for $x+y=2$ if $2\Re \nu <\Re \alpha+\Re \beta$).

Now, $S_{q_0}^{\alpha,\beta,0}(x,y)$ can be obtained taking $\nu \to 0$ in the above formula: on one hand, after writing the hypergeometric function as a power series and looking for a hypergeometric representation of the resulting limit, it turns out that
\begin{multline*}
   \lim_{\nu \to 0}
   \frac{1}{\nu}
   \frac{x^{-2\nu} \Gamma(\alpha+1)}{\Gamma(\nu+1) \Gamma(\alpha-\nu+1)}
   \left(-1 + \FF{\nu-\alpha,\nu}{\beta+1}{\frac{y^2}{x^2}}
   \right)
   \\
   = - \frac{\alpha}{\beta+1} \frac{y^2}{x^2}
   \,\tresFFdos{1,1,1-\alpha}{2,\beta+2}{\frac{y^2}{x^2}}.
\end{multline*}
On the other hand, L'H\^opital's rule gives
\[
   \lim_{\nu\to 0}
   \frac{1}{\nu} \left(
   -1 + \frac{x^{-2\nu} \Gamma(\alpha+1)}{\Gamma(\nu+1) \Gamma(\alpha-\nu+1)}
   \right)
   = - 2 \log x + H_\alpha,
\]
where $H_\alpha = \gamma + \frac{\Gamma'(\alpha+1)}{\Gamma(\alpha+1)}$ is the harmonic number of order $\alpha$ (as usual $\gamma$ denotes the Euler constant).
Then, we conclude that
\begin{multline*}
   S_{q_0}^{\alpha,\beta,0}(x,y)
   \\
   = \frac{x^\alpha y^\beta}{2^{q_0}\Gamma(\alpha+1)\Gamma(\beta+1)}
   \biggl(
  -2 \log x + H_\alpha - \frac{\alpha}{\beta+1} \frac{y^2}{x^2} 
  \,\tresFFdos{1,1,1-\alpha}{2,\beta+2}{\frac{y^2}{x^2}}
   \biggr).
\end{multline*}

\subsubsection{The case $n=1$}

The identities~\eqref{lap2}, \eqref{sod2}, \eqref{rdc}, and~\eqref{phi0} lead to
\begin{multline*}
   S_{q_1}^{\alpha,\beta,\nu}(x,y)
   = \frac{\Gamma(\nu+1)^2 x^\alpha y^\beta}{2^{q_1-1}\Gamma(\alpha+1)\Gamma(\beta+1)}
   \Bigg(
   \frac{x^2}{2\nu(\alpha+1)} + \frac{y^2}{2\nu(\beta+1)} - \frac{1}{\nu^2-1}
   \\
   + \frac{x^{2-2\nu}\binom{\alpha}{\nu}}{2\nu(\nu-1)}
   \bigg(
   \frac{\FF{\nu-\alpha-1,\nu}{\beta+1}{\frac{y^2}{x^2}}}{\alpha+1-\nu}
   + \frac{\frac{y^2}{x^2} \,\FF{\nu-\alpha,\nu}{\beta+2}{\frac{y^2}{x^2}}}{\beta+1}
   \bigg)
   \Bigg),
\end{multline*}
valid for $2\Re \nu < 3+\Re \alpha+\Re \beta$, $\nu\neq 0,1$, and $0 < y\leq x$, $x+y < 2$ (also for $x+y=2$ if $2\Re \nu < 2+\Re \alpha+\Re \beta$).

Now, $S_{q_1}^{\alpha,\beta,0}(x,y)$ and $S_{q_1}^{\alpha,\beta,1}(x,y)$
follow taking limits as $\nu \to 0$ and $\nu \to 1$ in the above formula with the same kind of manipulations of the case $n=0$. We thus obtain that
\begin{multline*}
   S_{q_1}^{\alpha,\beta,0}(x,y)
   = \frac{x^\alpha y^\beta}{2^{q_1-2}\Gamma(\alpha+1)\Gamma(\beta+1)}
   \Bigg(
  2 - \frac{x^2}{(\alpha+1)^2} 
  \\
  - \frac{((\alpha+1)y^2+(\beta+1)x^2)(1+H_\alpha-2\log x)}{(\alpha+1)(\beta+1)}
  \\
  + \frac{y^2}{\beta+1} \,\tresFFdos{1,1,-\alpha}{2,\beta+2}{\frac{y^2}{x^2}}
  + \frac{\alpha y^4}{(\beta+1)(\beta+2)x^2} 
  \,\tresFFdos{1,1,1-\alpha}{2,\beta+3}{\frac{y^2}{x^2}}
   \Bigg)
\end{multline*}
and
\begin{multline*}
   S_{q_1}^{\alpha,\beta,1}(x,y)
   = \frac{x^\alpha y^\beta}{2^{q_1}\Gamma(\alpha+1)\Gamma(\beta+1)}
   \Bigg(
  \frac{x^2}{\alpha+1} + \frac{y^2}{\beta+1} 
  - \frac{-3+2H_\alpha-4\log x}{2}
  \\
  + \frac{\alpha}{\beta+1}\frac{y^2}{x^2} 
  \,\tresFFdos{1,1,1-\alpha}{2,\beta+2}{\frac{y^2}{x^2}}
   \Bigg).
\end{multline*}

\subsection{The case $n<0$}

Once we have determined the case $n \geq 0$, let us consider now the case when $n$ is a negative integer.

From \eqref{ec:parcialxi}, we get
\[
   \xi_{-1,\alpha,\beta,\nu}(x,y)
   = - \frac{2\alpha}{x} \frac{\partial}{\partial x} \xi_{0,\alpha-1,\beta,\nu}(x,y).
\]
An easy computation using \eqref{forn2} gives
\[
   \xi_{-1,\alpha,\beta,\nu}(x,y)
   = \frac{\Gamma(\alpha+1)}{4(\nu+1)^2\Gamma(\nu+1)\Gamma(\alpha-\nu)}
   \, x^{-2\nu-2} \, \FF{\nu-\alpha+1,\nu+1}{\beta+1}{\frac{y^2}{x^2}},
\]
from where the identity \eqref{sod22} for $n=-1$ follows easily using \eqref{lap2}. The identities \eqref{sod22} for an integer $n < -1$ can be proved similarly. The case $\nu\in\{0,1,\dots, n\}$ follows by passing to the limit.

\subsection{The one variable case}

We will need later the following Sneddon-Bessel series in one variable:
\begin{equation}
\label{lap1v}
   \Ss_n^{\alpha,\nu}(x)
   = \sum_{m=1}^\infty
   \frac{J_\alpha(x j_{m,\nu})}{j_{m,\nu}^{2n+\alpha-2\nu+2} J_{\nu+1}(j_{m,\nu})^2},
\end{equation}
where $0\leq x\leq 2$.

If we assume $2\Re\nu < 2n + 1/2 + \Re\alpha$, then the uniform convergence of the Sneddon-Bessel series \eqref{lap} holds also for $y=0$, hence the Sneddon-Bessel series \eqref{lap1v} arises after dividing the Sneddon-Bessel series \eqref{lap} by $y^\beta$ and taking $y\to 0$. This can be done in the identities \eqref{sod2i} and \eqref{sod22}. To this end, we have to compute the sequence
\[
   d_n^{\alpha,\nu} = \phi_n^{\alpha,\beta,\nu}(0), \quad n\geq 0.
\]
After some easy computations, using \eqref{hp1} and \eqref{hp2}, we arrive at
\[
   d_n^{\alpha,\nu}
   = \frac{\Gamma(\alpha+1)\Gamma(\nu-n)}{2^{2n}\Gamma(\nu+1)^2 \Gamma(n+1+\alpha-\nu)},
   \quad n\geq 0.
\]
Let us define the polynomial $\delta_n^{\alpha,\nu}(x)$ (of degree $2n$), $n\geq 0$, by the generating function
\begin{equation}
\label{gna}
   \frac{\Phi_{\alpha}(xz)}{\Phi_\nu(z)^2}
   = \sum_{n=0}^\infty \delta_n^{\alpha,\nu}(x)z^{2n}.
\end{equation}
This generating function allows the explicit computation of the polynomials $\delta_n^{\alpha,\nu}(x)$ recursively from the Taylor coefficients of the functions $\Phi_\alpha$ and~$\Phi_\nu$.

If we write
\begin{equation}
\label{coed1v}
   \delta_n^{\alpha,\nu}(x) = \sum_{j=0}^n A_{j,n}^{\alpha,\nu} x^{2j} ,
\end{equation}
then for $\nu \notin \{0,1,\dots, n\}$, $2\Re \nu<2n+1/2+\Re \alpha$ and $0< x< 2$
(and also for $x=2$ if $2\Re \nu <2n-1/2+\Re \alpha$), we have
\begin{multline}
\label{sod2i1v}
   \Ss_{n}^{\alpha,\nu}(x)
   = \frac{\Gamma(\nu+1)^2 x^\alpha }{2^{\alpha-2\nu+2}\Gamma(\alpha+1)}
   \bigg(
   \frac{\binom{\alpha}{\nu}\Gamma(\alpha-\nu+1)\Gamma(\nu-n)}
   {2^{2n}\Gamma(\nu+1)\Gamma(n+1+\alpha-\nu)}
   \, x^{2n-2\nu}
   \\
   - \sum_{j=0}^n \frac{A_{j,n}^{\alpha,\nu}}{j+\nu - n} x^{2j}
   \bigg),
\end{multline}
if $n\geq 0$, and
\[
   \Ss_{n}^{\alpha,\nu}(x)
   = \frac{\Gamma(\nu-n) }{2^{2n+\alpha-2\nu+2}\Gamma(n+\alpha-\nu+1)}
   \, x^{\alpha-2\nu+2n},
\]
if $n<0$.

For instance, for $n=0$ we get
\[
   \Ss_0^{\alpha,\nu}(x)
   = \frac{\Gamma(\nu+1)^2 x^\alpha}{2^{2+\alpha-2\nu}\nu\Gamma(\alpha+1)}
   \left(- 1 + \binom{\alpha}{\nu}x^{-2\nu} \right),
\]
assuming $0 < x < 2$, $2\Re \nu < \frac{1}{2} + \Re \alpha$, and $\nu \neq 0$ (also for $x=2$ if $ 2\Re \nu < \Re \alpha$). And for $\nu=0$ and $-\frac{1}{2} < \Re \alpha$,
\[
   \Ss_0^{\alpha,0}(x)
   = \frac{x^\alpha }{2^{\alpha+1} \Gamma(\alpha+1)}
   \Bigl(
   -\log x + \frac{1}{2} H_\alpha
   \Bigr)
\]
which was previously computed using a different method in~\cite[(4)]{DPVa}.

And for $n=1$, the corresponding Sneddon-Bessel series is
\[
   \Ss_1^{\alpha,\nu}(x)
   = \frac{\Gamma(\nu+1)\Gamma(\nu-1)x^\alpha}{2^{3+\alpha-2\nu}\Gamma(\alpha+1)}
   \left(
   \frac{(\nu-1)x^2}{2(\alpha+1)} - \frac{\nu}{\nu+1}
   + \frac{\binom{\alpha}{\nu}x^{2-2\nu}}{2(\alpha+1-\nu) }
   \right),
\]
assuming $0 < x < 2$, $2\Re \nu < \frac{5}{2} +\Re \alpha$, and $\nu \neq 0,1$.

The cases $\nu = 0,1$ can be deduced taking limits as $\nu\to 0$ and $\nu\to 1$, respectively. As a result,
\[
   \Ss_1^{\alpha,0}(x)
   = \frac{x^\alpha}{2^{\alpha+3} \Gamma(\alpha+2)} \biggl(
   x^2 \log x - \frac{1 + H_{\alpha+1}}{2}\, x^2 + \alpha+1
   \biggr)
\]
for $-\frac{5}{2} < \Re \alpha$, and
\[
   \Ss_1^{\alpha,1}(x)
   = \frac{x^\alpha }{2^{\alpha+1} \Gamma(\alpha+1)}
   \biggl(
   - \log x + \frac{-3+2H_\alpha}{4} + \frac{1}{2(\alpha+1)} \,x^{2}
   \biggr)
\]
whenever $-\frac{1}{2} < \Re \alpha$.

\section{Extending the Kneser-Sommerfeld expansion}
\label{ksex}

In this section we use the identities \eqref{sod2} to prove some extensions of the Kneser-Sommerfeld expansion~\eqref{ks}.

For the sake of completeness, we first prove the Kneser-Sommerfeld expansion \eqref{ks}.
In terms of the functions $\Phi_\nu$ defined by \eqref{def:Phibeta}, the identity to be proved is
\begin{multline}
\label{ksss}
   \sum_{m=1}^\infty
   \frac{\Phi_\nu(x j_{m,\nu})\Phi_\nu(y j_{m,\nu})}
   {j_{m,\nu}^2(j_{m,\nu}^{2}-z^2) \Phi_{\nu+1}(j_{m,\nu})^2}
   \\ =
   \frac{\pi J_\nu(y z) \big(Y_\nu(z)J_\nu(xz)-J_{\nu}(z)Y_\nu(xz)\big)}
   {16(\nu+1)^2 (xy)^{\nu} J_\nu(z)}.
\end{multline}
Write $\varphi(x,y,z)$ and $\psi(x,y,z)$ for the left and right hand sides of~\eqref{ksss}, respectively. Using the geometric series, we can write
\[
   \varphi(x,y,z)
   = \sum_{m=1}^\infty
   \frac{\Phi_\nu(x j_{m,\nu})\Phi_\nu(y j_{m,\nu})}
   {j_{m,\nu}^{2} (j_{m,\nu}^{2}-z^2) \Phi_{\nu+1}(j_{m,\nu})^2}
   = \sum_{n=0}^\infty \xi_{n,\nu,\nu,\nu}(x,y) z^{2n},
\]
where the function $\xi_{n,\nu,\nu,\nu}(x,y)$ is defined by~\eqref{defxixy}.

Consider now the partial differential equation
\begin{equation}
\label{kspx}
   - \frac{\Phi_\nu(xz)\Phi_\nu(y j_{m,\nu})}{8(\nu+1)^2\Phi_\nu(z)^2}
   = 2\nu U(x,y,z) - z \frac{\partial U}{\partial z}(x,y,z)
   + x \frac{\partial U}{\partial x}(x,y,z)
   + y \frac{\partial U}{\partial y}(x,y,z).
\end{equation}
On the one hand, using the partial differential equation \eqref{eddp} for $\xi_n(x,y)$ and \eqref{jodr}, we get that $\varphi(x,y,z)$ satisfies the partial differential equation \eqref{kspx}. On the other hand, it is a matter of computation to check that the function $\psi(x,y,z)$ satisfies the partial differential equation \eqref{kspx} as well.
Hence, we deduce that
\[
   \varphi(x,y,z) - \psi(x,y,z) = x^{-2\nu} \rho(y/x,z/x),
\]
for certain two-variable function $\rho$. But the definition of $\varphi$ and $\psi$ as both sides of \eqref{ksss} shows that
\[
   \varphi(1,y,z) = \psi(1,y,z) = 0,
\]
so that $\rho=0$ and the identity \eqref{ksss} holds.

For $\Re \beta> \Re \nu>-1$, the identity \eqref{ksee} follows by applying the integral transform $T_{\beta,\nu}$ defined by \eqref{itc} acting in the variable $y$ to both sides of the Kneser-Sommerfeld expansion \eqref{ks} and using Sonin's formula \eqref{eq:sonin}. With a standard argument of analyticity, the identity \eqref{ksee} extends to $-1 < \Re \nu < \Re \beta + 1$.

If $\Re \nu<-1$, we can take a positive integer $h$ satisfying $\Re \nu > - h/2 - 1$. When $\Re\beta > \Re\nu + h$, using the integral transform $T_{\beta,\nu,h}$ defined by \eqref{itch}, we prove the identity \eqref{ksee} for $\Re \beta > \Re \nu + h$, and using an argument of analyticity for $\Re\nu<\Re \beta+1$.

The Kneser-Sommerfeld expansion has the following one variable version:
\begin{equation}
\label{ks1}
   \sum_{m=1}^\infty
   \frac{j_{m,\nu}^\nu J_\nu(x j_{m,\nu})}{(j_{m,\nu}^{2}-z^2) J_{\nu+1}(j_{m,\nu})^2}
   = \frac{\pi z^\nu}{4 J_\nu(z)} \big(Y_\nu(z)J_\nu(xz)-J_\nu(z)Y_\nu(xz)\big),
\end{equation}
valid for $\Re \nu<1/2$ (it follows easily dividing the identity \eqref{ks} by $y^\nu$ and then taking limit as $y\to 0$).

Now, the well-known properties of the Bessel function of the second kind $Y_\nu$ allow us to rewrite \eqref{ks1} as
\begin{equation}
\label{ks1r}
   \sum_{m=1}^\infty
   \frac{j_{m,\nu}^\nu J_\nu(x j_{m,\nu})}{(j_{m,\nu}^{2}-z^2) J_{\nu+1}(j_{m,\nu})^2}
   = \frac{\pi z^\nu}{4\sin(\pi\nu)J_\nu(z)}
   \big(J_\nu(z)J_{-\nu}(xz)-J_{-\nu}(z)J_\nu(xz)\big)
\end{equation}
(as usual, if $\nu=n$ is a nonnegative integer, the function on the right can be understood as the limit as $\nu\to n$).

We finish this paper proving the following extension of the identity~\eqref{ks1r}:
\begin{multline}
\label{ks1re}
   \sum_{m=1}^\infty
   \frac{j_{m,\nu}^{2\nu-\alpha} J_\alpha(x j_{m,\nu})}
   {(j_{m,\nu}^{2}-z^2) J_{\nu+1}(j_{m,\nu})^2}
   =
   \frac{\pi }{4\sin(\pi\nu)J_\nu(z)}
   \\
   \times \Bigg(
   \frac{x^{\alpha-2\nu} J_\nu(z) \,\unoFFdos{1}{-\nu+1,\alpha-\nu+1}{-\frac{(xz)^2}{4}}}
   {2^{\alpha-2\nu}\Gamma(-\nu+1)\Gamma(\alpha-\nu+1)}
   - z^{2\nu-\alpha} J_{-\nu}(z) J_\alpha(xz)
   \Bigg),
\end{multline}
valid for $2\Re \nu < \Re \alpha + 1/2$, $\alpha - \nu + 1 \neq 0,-1,-2,\dots$, $\nu\notin \ZZ$ and $0< x\leq 2$ (for $\nu=n\in \NN$, we can extend it passing to the limit $\nu\to n$).

For $\Re \nu < 1/2$ and $2 \Re\nu < \Re\alpha + 1/2$, the proof is similar to that of identity \eqref{ksee}. For $\Re\alpha > \Re\nu$ and $-1 < \Re \nu < 1/2$, the identity \eqref{ks1re} follows by applying the integral transform $T_{\alpha,\nu}$ defined by \eqref{itc} to both sides of the one variable Kneser-Sommerfeld expansion \eqref{ks1r}: in the left hand side we use Sonin's formula \eqref{eq:sonin} and in the right hand side the identity \eqref{eaq} applied to the power expansion of $J_{-\nu}(xz)$.
Using a standard argument of analyticity, the identity \eqref{ks1re} extends to $-2 < 2 \Re \nu < \Re \alpha + 1/2$. If $\Re \nu<-1$, we can take a positive integer $h$ satisfying $\Re \nu > -h/2-1$ and use the integral transform $T_{\alpha,\nu,h}$ defined by~\eqref{itch}.

In order to extend the identity \eqref{ks1r} to $2\Re \nu < \Re \alpha + 1/2$, $\alpha - \nu + 1 \neq 0,-1,-2,\dots$, $\nu \notin \ZZ$ and $0< x\leq 2$, we proceed as follows. First of all, since both sides of the identity \eqref{ks1r} are analytic functions of $z$, it would be enough to prove \eqref{ks1r} for $|z|$ small enough. To this end, let us find suitable bounds for the Sneddon-Bessel series \eqref{lap1v}. Looking at \eqref{sod2i1v}, for each $\nu$-compact set $K\subset \CC\setminus \ZZ$ we have
\begin{equation}
\label{spmu1}
   \bigg|
   \frac{\Gamma(\nu+1)\binom{\alpha}{\nu}\Gamma(\alpha-\nu+1)\Gamma(\nu-n)x^{\alpha+2n-2\nu}}
   {2^{\alpha-2\nu+2+2n}\Gamma(\alpha+1)\Gamma(n+1+\alpha-\nu)}
   \bigg|
   \leq c_{\alpha,K},
\end{equation}
whith a positive constant $c_{\alpha,K}$ depending only on $\alpha$ and the $\nu$-compact $K$; this follows from the fact that $0 < x \leq 2$ and
\begin{align*}
   \Gamma(\nu-n) &= \frac{\Gamma(\nu+1)}{\nu(\nu-1)\cdots(\nu-n)},
   \\
   \Gamma(n+1+\alpha-\nu) &= (n+\alpha-\nu) \cdots (\alpha-\nu+2)(\alpha-\nu+1) \Gamma(\alpha-\nu+1).
\end{align*}
Now let us consider the analytic function
\[
   F_{j,\nu}(z) = \frac{z^{2j}}{\Phi_\nu(z)^2}.
\]
for each $j \geq 0$. There exists some constant $C_K$ depending on the $\nu$-compact $K$ such that
\[
   |F_{j,\nu}(z)| \leq C_K,
\]
on the circle $|z| = 1/2$. Then, Cauchy's integral formula gives
\[
   \bigg| \frac{F_{j,\nu}^{(2n)}(0)}{(2n)!} \bigg| \leq C_K 2^{2n}.
\]
Using again Cauchy's integral formula, together with \eqref{coed1v} and \eqref{gna}, it follows that
\begin{equation}
\label{mtc1}
   | A^{\alpha,\nu}_{j,n}|
   = \bigg|
   \frac{\Phi_\alpha^{(2j)}(0)F_{j,\nu}^{(2n)}(0)}{(2j)!\,(2n)!}
   \bigg|
   \leq d_\alpha C_K 2^{4n}.
\end{equation}
Inserting the estimates \eqref{spmu1} and \eqref{mtc1} in \eqref{sod2i1v} proves that
\begin{equation}
\label{mtc2}
   | \Ss^{\alpha,\nu}_{n}(x) | \leq e_{\alpha,K} 2^{6n},
\end{equation}
where $e_{\alpha,K}$ is a constant depending only on $\alpha$ and the $\nu$-compact $K$. Now, the power series expansion of $(j_{m,\nu}^{2}-z^2)^{-1}$ and the definition~\eqref{lap1v} lead to
\begin{align}
\label{esc}
   \sum_{m=1}^\infty
   \frac{j_{m,\nu}^{2\nu-\alpha} J_\alpha(x j_{m,\nu})}{(j_{m,\nu}^{2}-z^2) J_{\nu+1}(j_{m,\nu})^2}
   = \sum_{n=0}^\infty \Ss_n^{\alpha,\nu}(x) z^{2n}.
\end{align}
To be precise: the estimate \eqref{mtc2} shows that
if $2\Re\nu < \Re\alpha + 1/2$, $\alpha-\nu+1\neq 0,-1,-2,\dots$, $\nu\notin \ZZ$, $0< x\leq 2$, and $|z| < 1/2^3$, the identity \eqref{esc} holds and the right hand side is an analytic function of $\nu$. Since we have already proved that the identity \eqref{ks1re} for $-1 < \Re \nu < 1/2$ and its right hand side is also an analytic function of $\nu$, we deduce that the identity \eqref{ks1re} holds indeed for $2\Re \nu < \Re \alpha + 1/2$, $\alpha - \nu + 1 \neq 0,-1,-2,\dots$, $\nu\notin \ZZ$, $0< x\leq 2$.



\end{document}